\documentclass[sn-mathphys,Numbered]{sn-jnl_arx}

\usepackage{graphicx}%
\usepackage{multirow}%
\usepackage{amsmath,amssymb,amsfonts}%
\usepackage{amsthm}%
\usepackage{mathrsfs}%
\usepackage[title]{appendix}%
\usepackage{xcolor}%
\usepackage{textcomp}%
\usepackage{manyfoot}%
\usepackage{booktabs}%
\usepackage{algorithm}%
\usepackage{algorithmicx}%
\usepackage{algpseudocode}%
\usepackage{listings}%
\usepackage{amscd}
\usepackage[all,cmtip]{xy}
\newtheorem*{theorem-a}{Theorem A}%
\newtheorem*{theorem-b}{Theorem B}%
\newtheorem*{theorem-c}{Theorem C}%
\newtheorem*{theorem-d}{Theorem D}%

\newtheorem{theorem}{Theorem}

\newtheorem{prop}[theorem]{Proposition}%
\newtheorem{example}{Example}%
\newtheorem{rem}{Remark}%
\newtheorem{lemma}{Lemma}%
\newtheorem{cor}{Corollary}%

\newtheorem{definition}{Definition}%

\begin{document}

\title[Tiling spaces are covering spaces over irrational tori]{Tiling spaces are covering spaces over irrational tori}


\author[1]{\fnm{Dar\'io} \sur{Alatorre}}\email{dario@im.unam.mx}

\author[2]{\fnm{Diego} \sur{Rodr\'iguez-Guzm\'an}}\email{diego.rodriguez@academicos.udg.mx}

\affil*[1]{\orgdiv{Instituto de Matem\'aticas}, 
\orgname{Universidad Nacional Aut\'onoma de M\'exico}, 
\orgaddress{\street{\'Area de la Investigaci\'on Cient\'ifica, Circuito Escolar, Ciudad Universitaria}, \city{Ciudad de M\'exico}, \postcode{04510}, \country{M\'exico}}}

\affil[2]{\orgdiv{Departmento de Matem\'aticas}, 
\orgname{Centro Universitario de Ciencias Exactas e Ingenier\'ias, Universidad de Guadalajara}, 
\orgaddress{\street{Blvd. Marcelino Garc\'ia Barrag\'an}, \city{Guadalajara}, \postcode{44430}, \state{Jalisco}, \country{M\'exico}}}


\abstract{We study tiling spaces in the diffeological context.  
We prove some basic diffeological properties for tiling spaces and analyze two different 
fiber bundle structures of tiling spaces over irrational tori.
We use the diffeological classification of irrational tori which captures their arithmetical escence 
in order to inherit the diffeological equivalence in the context of one-dimensional tiling spaces. 

}



\maketitle

\section{Introduction}\label{sec1}

Our motivation comes from the naive question: what is the aperiodic analog of the orbifolds that characterize the symmetries of periodic tilings of the plane? 
Conway's Magic Theorem 
\cite{conway2008} classifies the symmetries of periodic tilings of the plane through the orbifolds obtained as 
quotients $\mathcal O= \mathbb R^2/G$, where $G$ is the symmetry group of the tiling. In the crystallographic case, $G$ has a subgroup of translations isomorphic to $\mathbb Z^2$. 
Thus the aforementioned orbifolds are quotients of the torus $\mathbb T^2$ by a discrete subgroup of $G$. 
In the aperiodic one-dimensional case, however, one will find that a plausible analog of the translations subgroups $\mathbb Z$ could be 
(a subset of) $\mathbb Z + \alpha\mathbb Z$, with $\alpha$ an irrational number.
 So one should look at $\mathbb{T}_\alpha := \mathbb R/(\mathbb Z + \alpha\mathbb Z)$, which is quite a strange object, its natural topology is coarse.
  It is nevertheless non-trivial when endowed with a structure of diffeological space, 
  it is called the \emph{irrational torus} of slope $\alpha$. It was first studied by Iglesias-Zemmour and Donato in 
  \cite{PIZ-Donato} (see also \cite{PIZ85,PIZ20}). 
On the other hand, another possible analog for $\mathcal O$ could be the tiling space $\Omega$, also known as the \emph{continuous hull}
 of a tiling. 
 For instance, the tiling space of a periodic tiling is a torus obtained by identifying translations. 
 If further symmetries are to be considered then one will end up with Conway's orbifolds.  

\medskip

Our approach is to define a diffeological space structure on $\Omega$ and analyze its 
relationships with the irrational torus $\mathbb{T}_\alpha$. 
We verify some basic diffeological properties for $\Omega$ and explore two different 
fiber bundle constructions of $\Omega$ over $\mathbb{T}_\alpha$.  
The first one is a variant of the fiber bundle over the torus with Cantor set fibers first constructed by Williams \cite{williams2001} 
and later on generalized by Williams himself and Sadun \cite{Sad}. 
  This variant turns out to be more natural in the sense that there is no need to deform 
  the tiles in order to construct it, so it holds up to stronger equivalence -- we include a discussion 
  of different sorts of equivalence in geometrical and combinatorial contexts. This result is stated as
\medskip

\begin{theorem-a}
  Let $\Omega$ be a space of repetitive tilings of $\mathbb R^d$. 
  Then $\Omega$ is a covering space over $ \mathbb R^d / R_v = \mathbb{T}_v$, where $R_v$ is the return module of a vertex $v$.   
\end{theorem-a}
    
\medskip
  
The second bundle is closely related to the irrational flows on tori. We exploit the diffeological formalism in this case 
by using the classification of irrational tori in order to spot an equivalence of symbolic sequences 
called \emph{strong orbital equivalence} which is inherited by tiling spaces:
  
\medskip

\begin{theorem-b}
  Let $\Omega_\alpha$ be a tiling space of canonical projection of dimension 2 to 1 and irrational 
  slope $\alpha$. Then $\Omega_\alpha$ is an $\mathbb R$-principal bundle over 
  the irrational torus with two origins 
  $\mathbb{T}_{\alpha}^{\bullet}$. 

\end{theorem-b}

\medskip

\begin{theorem-c}
  Let $\Omega_\alpha$, $\Omega_\beta$ be tiling spaces satisfying the conditions of Theorem B. Then the 
  following conditions are equivalent:
  \item[i)] $\Omega_\alpha$ and $\Omega_\beta$ are strong orbit equivalent
  \item[ii)] Their associated sturmian spaces $X_\alpha, X_\beta$ are strong orbit equivalent 
  \item[iii)] $\mathbb T_\alpha$ and $\mathbb T_\beta$ are diffeomorphic
  \item[iv)] $\alpha$ and $\beta$ are conjugated under the modular group $\mathrm{GL}(2,\mathbb{Z})$.

\end{theorem-c}

\medskip

These results imply some other nice properties. On the one hand, the existence of such a bundle 
allows us to define and compute a
fundamental group for tiling spaces thanks to the long exact sequence
of homotopy groups -- this appears as Corollary \ref{cor:gfomega}.
On the other, they provide straightforward arguments for concluding some other known results about 
tilings belonging to the intersection of the classes of substitution and projection tilings
-- stated in Corollary \ref{cor:whithintheclass}. 




\bigskip

This kind of problems have been extensively studied from different approaches as 
topological \cite{kell-top,barge2007rigidity, sadun-julien,ormes-sadun, sadun08}, 
dynamical \cite{Robinson,alcaldecuesta,ArnouxExample,ArnouxFischerScenery,durand-mass-soe},
and algebraic \cite{brix23, connes, forrest-hunton, forr-hunt, GPS,kell-put, rieffel}, among many others. 
Our results 
 provide both an alternative 
approach to previously known results as well as a basis and new techniques 
for further research towards an understanding of the smooth structures in tiling spaces and related objects. 
It is worth mentioning that there are some recent advances about the links between diffeological 
theory and non-commutative geometry in more general settings \cite{PIZ-Dif-NCG-1,PIZ-Dif-NCG-2,verj-lupe}.  
Our work could provide a series of examples for these recent advances. 

\medskip

The article is organized as follows. In Section \ref{sec:prelim} we introduce the basic notions of both tiling spaces and diffeology,
we provide in each case several definitions and state the results that we use throughout the manuscript. 
In Section \ref{sec:diff-TS} we define and study a diffeological version of tilings spaces, 
we construct the first of the mentioned bundles $\hat p: \Omega\to \mathbb{T}_\alpha$ and investigate its \emph{structure groupoid}, 
in order to prove Theorem A. 
Section \ref{section:diff-sturm} is devoted to investigate the strong orbital equivalence between tiling spaces and to the proof of the remaining results.

\section{Preliminaries}
\label{sec:prelim}


\subsection{Background on Tiling Spaces}

\subsubsection{Tilings}

\begin{definition}
\label{def:tiling} 
{\upshape 
 A \emph{tiling} $T=\{t_i\}_{i=1}^{\infty}$ of $\mathbb R^d$ is a family of closed subsets $t_i$ of $\mathbb R^d$ called \emph{tiles} satisfying the following conditions:
\begin{itemize}
\item[i)] Each $t_i$ is a translated copy of some element of a finite set of \emph{proto-tiles} $\mathcal P =\{\tau_1,\ldots, \tau_k\}$.
\item[ii)] Each $\tau_i$ is homeomorphic to a closed $d$-disk.
\item[iii)] $\cup_{i=1}^{\infty} t_i= \mathbb R^d$.
\item[iv)] For each $i,j$, if $i\neq j$ then $\mathrm{int}(t_i)\cap \mathrm{int}(t_j)=\emptyset$.
\end{itemize}
}
\end{definition}

A tiling $T$ is \emph{periodic} if its symmetry group has an infinite cyclic subgroup and it is \emph{aperiodic} in
any other case. 
There are at least three methods for constructing aperiodic tilings: substitution, projection, and local rules. We will consider only the first two classes. 

A \emph{substitution} rule is a function from the set of proto-tiles onto the set of \emph{patches}. 
By patch we mean a finite and connected union of tiles. 
It consists of a subdivision rule for the proto-tiles followed by a stretching by a linear factor $\lambda > 1$, 
resulting in a patch of the original set of proto-tiles. By iterating a substitution rule over a finite initial configuration of tiles, arbitrarily large portions of an euclidean space can be covered and thus a tiling can be constructed. 

The \emph{projection} method, also known as the \emph{cut-and-project} method, consists in selecting a strip from a higher dimensional periodic tiling and projecting it onto a totally irrational subspace. 

Each of these methods has its own structural implications. We will explain them further in the following sections.

\subsubsection{Tiling spaces}

There is a natural action by translation of $\mathbb R^d$ on every tiling $T$ given by
$$
T+x=\{t+x: t\in T\}
$$
for $x\in \mathbb R^d$. Let 
$$
O(T)=\{T+x: x\in \mathbb R^d\}
$$
be the orbit of $T$ under this action. 
Define a metric  $d$ on $O(T)$ where two tilings $T,T'$ are $\varepsilon$-close to each other if there are vectors $x, x'$ in $\mathbb R^d$ with norm less than $\frac{\varepsilon}{2}$ such that
$$
B_{\frac{1}{\varepsilon}}(0)\cap (T+x) = B_{\frac{1}{\varepsilon}}(0)\cap (T'+x').$$
This means that two tilings are close to each other if they coincide in a big euclidean ball around the origin up to a small translation. 
\bigskip

\begin{definition}
\label{def:hull}
{\upshape  Define the {\itshape tiling space} (also known as the \emph{hull} of $T$) $\Omega_T$ of the tiling $T$ as the completion of $O(T)$ for the metric $d$ . 
}
\end{definition}

\medskip

By definition, $\Omega_T$ contains every tiling $T'$ which is locally the same as $T$, 
meaning that every patch of $T'$ may be found on $T$. 
In general, a tiling space $\Omega$ is a translation invariant set of tilings which is complete in the metric $d$.

Two tilings $T$ and $T'$ are {\itshape locally indistinguishable} $(LI)$ if every patch of $T$ is a 
patch of $T'$ and vice versa. 
 Local indistinguishability is obviously an equivalence relation whose equivalence clases $LI$
 are contained in the hull $\Omega_T$. In fact, the $LI$ class $LI(T)$ of a tiling $T$ is a dense 
 subset of $\Omega_T$. In particular, if $T'$ is $LI$ from $T$ then $\Omega_T = \Omega_{T'}$.
 Moreover, under additional assumptions, the set of $LI(T)$ classes coincides with the hull $\Omega_T$, as stated in Proposition \ref{repLI}.

\smallskip

  A substitution rule $\omega$ is called \emph{primitive} if there is an integer number $N$
  such that every proto-tile appears in the $N$-th iteration of the 
  substitution rule applied to every other proto-tile. 
  Primitive substitutions generate unique $LI$ classes, 
  thus every tiling generates the same hull and so we may speak about 
  a tiling space $\Omega_\omega$ associated to $\omega$. 

  \smallskip

A tiling $T$ has {\itshape finite local complexity (FLC)} if for every real number $R>0$ 
there is a finite number of patches, up to translation, that fit inside the ball of radius $R$; 
and $T$ is called {\itshape repetitive} if for every patch $P$ of $T$ there is a radius $R>0$ such that for every $x\in \mathbb R^d$ the ball $B_R(x)$ contains a translated copy of $P$.

\smallskip

  The following is a folklore result that may be found in \cite{baake2013aperiodic}, 
  together with the previous discussion. 
 
  \bigskip
  \begin{prop}
 \label{repLI}
 Let $T$ be a $FLC$ tiling. Then $\Omega_T$ is compact and the following conditions are equivalent:
 \begin{enumerate}
 \item $T$ is repetitive,
 \item The dynamical system $(\Omega_T, \mathbb R^d) $ is minimal (every orbit is dense),
 \item $LI(T)$ is closed,
 \item $LI(T) =  \Omega_T$. 
 \end{enumerate}
 \end{prop}

 \smallskip

This fairly large class of tilings produces tiling spaces with nice structures.
 The hulls $\Omega$ of such tilings are locally homeomorphic to cartesian products of euclidean balls $D$ 
 and totally disconnected sets $K$ (Cantor sets under additional assumptions); 
 they are a rich source of examples for hyperbolic dynamics, they are known to be laminated spaces \cite{Ghys}, 
 inverse limits of branched manifolds \cite{AndPut,sadun03inverse,BBG06} and fiber bundles over tori \cite{williams2001,Sad}. 
 In the following paragraphs, we will state the last couple of theorems of this list and 
 briefly explain the intuition behind them.
 For a deeper introduction to the topology of tiling spaces see \cite{sadun08}. 


\medskip

\begin{theorem}[\cite{AndPut, sadun03inverse,BBG06}]
 
 \label{thm:APinverselimits}
 Tiling spaces are inverse limits of branched manifolds.

\end{theorem}

\smallskip

\noindent
{\bfseries Intuition behind Theorem \ref{thm:APinverselimits}.}
The iteration of a substitution rule assures that arbitrarily large euclidean balls can be covered.
This procedure 
can be encoded in the structure of an inverse limit. 
Regard each iteration of a substitution rule both as a finite pattern and a cylinder set, meaning 
that the finite pattern also represents all the tilings showing that pattern in some particular position 
(the canonical choice being around the origin). This cylinder set is identified with a branched 
manifold in the following way: consider the disjoint union of the set of prototiles and identify 
two of them by their borders if they appear next to each other somewhere in the tiling. 
Let $\Omega_0$ be the space obtained after the identification. 
It is a $CW$-complex, and it is actually a branched manifold. 
Then let $\Omega_n$ be the branched manifold obtained by applying $n$ times
the substitution rule to $\Omega_0$ (see Fig. \ref{fig:branched}). 
The theorem by Anderson and Putnam \cite{AndPut} asserts that the tiling space $\Omega$ 
is topologically conjugate to the inverse limit of ${\Omega_n, f_n}$ where the $f_n$ 
maps consist in forgetting the extra information from $\Omega_n$ to fit in $\Omega_{n-1}$. 
In other words, we can regard $\Omega$ as a subset of the infinite product of branched manifolds 
$\Pi_{n\in\mathbb N}\Omega_n$, and 
$$
\Omega \cong \lim_{\leftarrow i}(\Omega_n, f_n) = \{ (\bar x_n) \in \Pi_{n\in\mathbb N} \Omega_n : f_{n+1}(\bar x_{n+1}) = \bar x_n, \forall n \in \mathbb N) \}. 
$$

\smallskip

This result was later generalized in \cite{sadun03inverse} and \cite{BBG06} 
for general tiling spaces with FLC without the need of a substitution rule. 

\medskip

\begin{rem}
  \label{rmk:forcetheborder}
  {\upshape 
  We omited here the technicality of whether the substitution rule \emph{forces the border} — 
   which essentially means that, by applying the substitution, the patches would 
   ``grow in every direction'', not as, for example, the chair tiling substitution which would
   typically grow over the first quadrant of the plane.  
   When a substitution does not force the border, there is an additional step in the 
    proof of the previous theorem called \emph{collaring}, 
    which results in a modified complex $\tilde\Omega_0$. See \cite{sadun08} for details.  
  }
\end{rem}

\medskip

\begin{figure}
    \centering
    \includegraphics[width=3.6in]{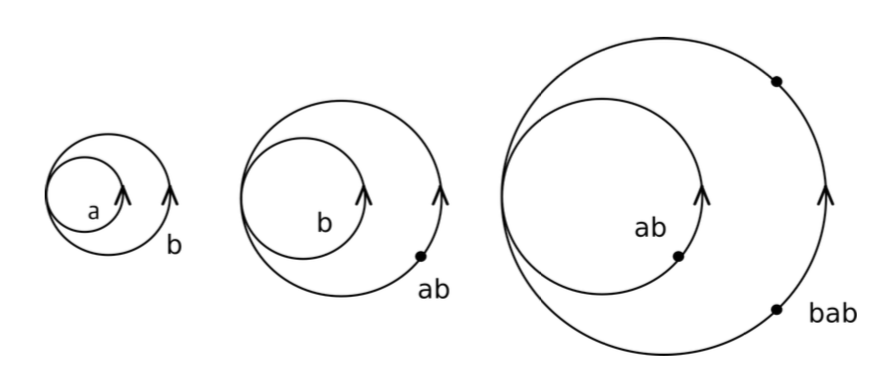}
    \caption{The branched 1-dimensional manifolds $\Omega_0$, $\Omega_1$ and $\Omega_2$ associated to 
    the subtitution rule $a\mapsto b$, $b\mapsto ab$.
    Note that this image is only for illustrative purposes but it is not precisely what is needed 
    -- we are omitting here a technicality called \emph{forcing the border}, see Remark \ref{rmk:forcetheborder}.
    }
    \label{fig:branched}
\end{figure}

\bigskip
\begin{theorem}[Sadun-Williams, 2001]
 \label{thm:SW-bundle}
 Tiling spaces are fiber bundles over tori with totally disconnected fiber. 
\end{theorem}

\smallskip

\noindent
{\bfseries Intuition behind Theorem \ref{thm:SW-bundle}.}
The first versions of this theorem are due to R. Williams \cite{williams2001} for Penrose, Ammann-Beenker and the stepped plane tiling spaces. It was generalized later on by himself and L. Sadun \cite{Sad} for every tiling space with FLC. 

The idea for constructing such a bundle uses the previous result and the canonical projection $p_0$ from the inverse limit (which can be identified with) $\Omega$ onto the branched manifold $\Omega_0$, 
and then constructing a map from $\Omega_0$ which identifies each proto-tile with a torus of the corresponding dimension. 
The construction of the latter map relies on the fact that tiles can be deformed in order for their vertices to have rational coordinates.
Such a deformation works only up to homeomorphism of tiling spaces. 

Regarding the fiber, in the cases of substitution or projection tilings, it is a Cantor set, 
whereas for periodic tilings it is a finite collection of points. 
 
\medskip

One of the results of the current article is the construction of a variant of this bundle in the diffeological 
setting.


\subsubsection{Equivalence of tiling spaces}
\label{ss:equivalence}

There are several different sorts of equivalence among tiling spaces. It is rather common to 
find at least three of them in the literature: 
\begin{itemize}
  \item Homeomorphism, in the usual topological sense. This is the weaker equivalence of this 
list. Homeomorphisms map translation orbits into translation orbits (provided the tilings have FLC), but not much more than that. 
  \item Topological conjugacy or simply conjugacy is stronger than homeomorphism. This equivalence preserves most of the dynamical properties of the spaces. 
  \item Mutual local derivability is the strongest of these equivalences. It means that it is possible to obtain 
  different tilings by local manipulation. For example, Penrose's Kite and Dart tilings are MLD with 
  Penrose's rhombi. 
\end{itemize}

\bigskip
\begin{example}
\label{ex:fibsust}
  {\upshape 
  
  Topological conjugacy is a rather rigid property among tilings. To exemplify this claim we 
  reproduce the following example from \cite{sadun08}. Consider  
  a one-dimensional tiling space $\Omega$ coming from the Fibonacci substitution
  $a\mapsto b$, $b\mapsto ab$, with tile lengths $|a| = 1$ and
  $|b| = \varphi = \frac{1+\sqrt{5}}{2}$. Let $a_n, b_n$ denote the patches (words) obtained by applying the substitution $n$
   times to $a$ and $b$, respectively. 
  Let $\Omega'$ be a suspension of the Fibonacci subshift with tile lengths
   $|a'| = |b'| = (2 + \varphi)/(1 + \varphi)$, and let  
   $a'_n$ and $b_n'$ the analogs of $a_n, b_n$. 
   The spaces $\Omega$ and $\Omega'$ are known to be topologically conjugate by a mapping which 
   preserves the sequences of the tiles. This happens because 
  the differences of the lengths $|a_n|-|a'_n|$ and $|b_n|-|b_n'|$ go to zero as $n$ tends to infinity. 

  }
\end{example}

\bigskip

The first kind of equivalence between tiling spaces that we will encounter in concordance with diffeological 
equivalences is the following: 

\smallskip

\begin{definition}
\label{def:oe}
{\upshape 
    Two tiling spaces $\Omega$ and $\Omega'$ are \emph{orbit equivalent} if there exist a 
    homeomorphism $h: \Omega \to \Omega'$ mapping the orbit $O(T)$ of $T$ in $\Omega$ onto the orbit 
    $O(h(T))$ of $h(T)$ in $\Omega'$ for every tiling $T\in\Omega$.
}     
\end{definition}

\bigskip

We will see in Section \ref{section:diff-sturm} that there is a stronger equivalence between tiling spaces 
that fits into the diffeological machinery: the so-called \emph{strong orbit equivalence}.


\subsubsection{Geometrical and combinatorial aspects of tiling spaces}
\label{sss:geocomb-ts}

\begin{definition}
\label{def:CPS}
{\upshape 
A {\itshape cut-and-project set} (or CPS) is a triple $(E, H, \mathcal L)$ where $E\cong \mathbb R^d$ is the
  {\itshape physical space} of the tiling, $H$ is a locally compact abelian group called {\itshape internal space},
    a lattice $\mathcal L\subset \mathbb R^d\times H$, and the canonical projections 
    $\mathrm{pr}_1: \mathbb R^d\times H \to \mathbb R^d$ and $\mathrm{pr}_2: \mathbb R^d\times H \to H$ 
    such that $\mathrm{pr}_1|_\mathcal L$ is injective and $\mathrm{pr}_2(\mathcal L)$ is dense in $H$. 
}
\end{definition}
\bigskip
\begin{figure}[htbp] 
   \centering
   \includegraphics[width=3.5in]{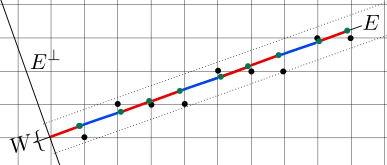} 
   \caption{Cut and project tiling from dimension 2 to 1.}
   \label{img:cps2-1}
\end{figure}

\smallskip

We will assume that $H = E^\perp = \mathbb R^n$, thus we regard $(\mathbb R^d, \mathbb R^n, \mathcal L)$ 
as a CPS from dimension $n+d$ to $d$. Given a window $W \subset H$, the CPS returns a point set: the projection onto $E$ 
of the subset $\mathcal{L}'\subset \mathcal{L}$ of lattice points whose projection $\pi_2(\mathcal{L}')$ onto $H$ lies within $W$. 
That is 
 $$\Lambda (W) = \mathrm{pr}_1 (\{y\in \mathcal L : \mathrm{pr}_2(y) \in W\}).$$
 We regard this point set as the set of vertices of a tiling. 
 The injectivity of $\pi_1|_\mathcal L$ implies the non-periodicity of the configuration.
 When the window $W$ is the projection of the unit cube of the lattice $\mathbb Z^{n+d}$ onto the internal space $\mathbb R^n$,  
 the CPS is called a {\itshape canonical projection}. 

 \medskip

 \begin{rem}
  \label{rem:cpsts}
{\upshape
The tiling space of a $(\mathbb R^d, \mathbb R^n, \mathcal L)$-CPS  is closely related to a torus of dimension $n+d$. 
We explain further the case of projection from dimension 2 to 1. 
Consider 
the line $L_\alpha=\{y=\alpha x\}$ with slope the irrational number $\alpha$. 
Thus, the $(L_\alpha,\mathbb{R},\mathbb{Z}^2)$-CPS of canonical projection generates aperiodic tilings $T(\alpha)$ 
with two tiles labelled with $0,1$ of lengths $\frac{1}{\sqrt{1+\alpha^2}},\frac{\alpha}{\sqrt{1+\alpha^2}}$, respectively. 
Observe that each point $p$ in $L_\alpha$ corresponds to an element of $O(T(\alpha))$ by the relation 
\begin{equation}
p=x\frac{(1,\alpha)}{\sqrt{1+\alpha^2}}\sim T(\alpha)-x
\end{equation} 
Note that there exists an ambiguity whenever the line $L_\alpha$ intersects a point with integer coordinates, 
say $(0,0)$ for simplicity,  
as both projections $\pi_2(0,1)$ and $\pi_2(1, 0)$ fall in the border of the canonical window $W$. 
Selecting either of them gives  us a different tiling, 
and these two tilings only differ in the tiles located on their origins, so they are at maximun distance in the 
tiling metric.  
 
Thus the tiling space $\Omega_\alpha=\Omega_{T(\alpha)}$ (Definition \ref{def:hull}) is homeomorphic to the 2-torus
\begin{equation}
\pi:\mathbb{R}^2\to\mathbb{R}^2/\mathbb{Z}^2=\mathbb{T}^2,
\end{equation} 
with a double solenoid $\mathcal S_{\alpha}'=\pi(L_\alpha)^{(0,1)}\cup\pi(L_\alpha)^{(1,0)}$ 
and every translation orbit 
is represented by the projection $\pi(y=\alpha x+\rho)$ in $\mathbb T^2$, with $\rho\in(0,1)$ . 
The double solenoid $\mathcal S_{\alpha}'$ comes from the ambiguity mentioned above.
See Section 1.4 in \cite{sadun08}, and \cite{forrest-hunton} for a more complete account.

For the line  $L_{\frac{1}{\varphi}}$ with $\varphi = \frac{1+\sqrt{5}}{2}$, 
the canonical projection tiling obtained by this method coincides with a resize of the Fibonacci tiling described 
in Example \ref{ex:fibsust}.
} 
\end{rem}

\bigskip

Given a tiling $T=\{t_i\}_{i\in\mathbb{Z}}$ of $\mathbb{R}$ we can identify its set of proto-tiles $\{\tau_j\}$  
with numbers $\{a_j\}$ so as to obtain a sequence $S=\{s_i\}_{i\in\mathbb{Z}}$ which takes values in $\{a_j\}$. 
The sequence $S$ is the \emph{combinatorial aspect} of the tiling $T$. 
This kind of sequences, endowed with the shift map 
\begin{equation}\label{function:shift}
\begin{array}{ccccc}
\sigma:&\{a_j\}^{\mathbb{Z}}&\to&\{a_j\}^{\mathbb{Z}}& \\
       & S & \mapsto&\sigma(S)&\quad\text{with}\quad \sigma(S)_i=s_{i+1},
\end{array}
\end{equation}
are studied as symbolic dynamical systems 
within the more general context of \emph{Cantor minimal systems} (see for example \cite{Putnam-CMS}). 
 
The symbolic sequences we are mainly interested in are the \emph{sturmian sequences}. In fact, 
the combinatorial aspect of the Fibonacci tiling (called in that context the Fibonacci word) is an example of
a sturmian sequence.

\medskip

\begin{definition}
\label{def:sturmian}
{\upshape 
  The \emph{sturmian sequence} of slope $\alpha$ and intercept $\rho \in [0,1)$ 
  is the sequence 
  $S=\{s_i \}_{i\in\mathbb Z}$ of 0 and 1 defined by 
  $s_n (\rho) = \lceil n\alpha + \rho \rceil  - \lceil (n-1)\alpha + \rho \rceil$. 
}
\end{definition}

\bigskip
 
The sturmian word of slope $\alpha$ and intercept $\rho$ can also be encoded as 
the cutting sequence\footnote{
The cutting sequence of a line $L_\alpha$ is constructed by writing 1  
anytime the line $L_\alpha$ crosses an integer horizontal line between the parallels lines 
$x=n-1$ and $x=n$ and 0 when it doesn't cross any integer horizontal line.} 
of the line $y=\alpha x + \rho$. 
In terms of their complexity, sturmian sequences are characterized as the words containing exactly $n+1$ different
subwords of lenght $n$. The literature on sturmian sequences is vast (see for example \cite{lothaire02}). 
We will only focus 
on the types of equivalences among sturmian sequences that turn out to be relevant in our 
context. 

\medbreak

\bigskip
\begin{definition}
\label{def:sturmianSystem}
{\upshape 
    Let $S_\alpha$ be a sturmian sequence. The \emph{sturmian system} $(X_\alpha, \sigma)$ asociated to 
    $S_\alpha = \{ s_j\}_{j\in\mathbb Z}$ 
    is defined as 
    $$
      X_\alpha := \overline{ \{ \{ s_j(\rho)\}_{j\in \mathbb Z} : \rho \in [0,1) \} } \subset \{0,1\}^\mathbb Z
    $$
    and the shift map $\sigma: X_\alpha \to X_\alpha$ (\ref{function:shift}).
}
\end{definition}

\medskip

Note the similitudes of the constructions of tiling spaces and sturmian spaces. We can regard the latter  
as ``symbolic'' or ``discretized'' tiling spaces.

\bigskip
\begin{definition}
\label{def:oecms}
{\upshape 
    Two sturmian spaces $(X,\sigma)$ and $(Y,\sigma)$ are \emph{orbit equivalent} if there exist a 
    homeomorphism $h: X \to Y$ mapping the orbit $O_{\sigma}(x)$ of $x$ under $\sigma$ in $X$ 
    onto the orbit $O_{\sigma}(h(x))$ under $\sigma$ of $h(x)$ in $Y$, for every $x\in X$.
}     
\end{definition}

\medskip
This implies that there is an unique \emph{cocycle map} $n:X\to \mathbb Z$ such that $h(\sigma(x)) = \sigma^{n(x)}(h(x))$.

\bigskip
\begin{definition}
\label{def:soecms}
{\upshape 
    
Two sturmian spaces $(X,\sigma)$ and $(Y,\sigma)$ are \emph{strong orbit equivalent}
if they are orbit equivalent and the cocycle map has at most one point of discontinuity. 
}
\end{definition}

\medskip
\begin{definition}\label{def:equivIrrat}
{\upshape 
We will say that two irrational number $\alpha$ and $\beta$ are \emph{equivalent}
 if there exists a matrix $M\in\mathrm{GL}(2,\mathbb{Z})$ such that
\begin{equation}
\alpha =\frac{c+\beta d}{ a+\beta b}=M\beta. 
\end{equation}
}
\end{definition}

\begin{lemma}\cite[Proposition 3.3]{brix23}\label{Lemma:EquivSturmian}
Let $(X_\alpha,\sigma)$ and $(X_\beta,\sigma)$ be two sturmian spaces. 
The following are equivalent:
\begin{itemize}

 \item[(1)] $\alpha\sim\beta$ (equivalence of irrationals);
 \item[(2)] $X_\alpha$ and $X_\beta$ are strong orbit equivalent.
 
\end{itemize}
\end{lemma}


\subsection{Background on Diffeology}

Diffeology is a generalization of differential geometry that includes manifolds and orbifolds, among other spaces. 
One of its attributes is that it is possible to deal with pathological spaces such as quotients by 
dense subspaces -- the case we are interested in. 

All of the content of this section may be found fully developed in \cite{PIZ13}. We will adopt most of its terminology. By \emph{euclidean domain} it should be understood an open subset of some euclidean space $\mathbb R^n$, and by  \emph{parametrization} any function from an euclidean domain to a set $X$. 

\subsubsection{Diffeological Spaces}
\label{ss:diff-sp}

\begin{definition}

    {\upshape Given a set $X$, a \emph{diffeology} $\mathcal D$ for $X$ is a set of  parametrizations $P:U\to X$, where $U$ is an open subset of some $\mathbb R^d$, subject to the following axioms: 
    \begin{itemize}
        \item[D1.] Every constant parametrization of every euclidean domain on $X$ belongs to $\mathcal D$.
        \item[D2.] If $P:U\to X$ is a parametrization such that every $u\in U$ has a neighborhood $V\subset U$ such that $P|_V\in \mathcal D$, then $P\in \mathcal D$. 
        \item[D3.] If $P:U\to X$ belongs to $\mathcal D$ and $f:V\to U$ is a smooth map between euclidean domains then $P\circ f \in \mathcal D$. 
    \end{itemize}
    A diffeological space is a pair $(X, \mathcal D)$ where $\mathcal D$ is a diffeology for the set $X$. The elements of $\mathcal D$ are called {\itshape plots}.
    }
    
\end{definition}  

\bigskip

\begin{example} 
\label{ex:standard-diff}
 
  {\upshape 

Every $\mathbb R^d$ is a diffeological space whose structure is inherited from that of a differentiable manifold. Consider the set of parametrizations 
    $$ \mathcal{C}^\infty_* = \bigcup_{n\in\mathbb{N}} \{P:U\rightarrow \mathbb{R}^d : U \subset \mathbb{R}^n \text{ and }  P \text{ is } C^\infty \}_n.
    $$
    The pair $(\mathbb{R}^d, \mathcal{C}^\infty_*)$ is a diffeological space and $\mathcal{C}^\infty_*$ is called \emph{the standard  diffeology} for $\mathbb{R}^d$.  In the same fashion, every differentiable manifold is a diffeological space. 
    
  }
\end{example}

\medskip

A function $f:X\to Y$ between diffeological spaces is {\itshape smooth} if, for every plot $P:U \to X$ of $X$, the composition $f \circ P$ is a plot of $Y$. 

\medskip

In a similar manner as in topology, diffeologies can be ordered and there is a minimum and a maximum 
with respect to this order. 
The minimum  is the the \emph{fine} or \emph{discrete} diffeology whose plots are all locally constant; 
and the maximum is the \emph{coarse} or \emph{indiscrete} diffeology and it is defined as the diffeology that contains every parametrization.  

\smallskip

Diffeological spaces and smooth maps between them form a category which we shall denote by $\bf{Diff}$. 
It is a well-behaved category in the sense that it is closed under taking quotients, subobjects, limits, colimits and spaces of functions $C^\infty(\star, \star)$, where $C^\infty$ denotes the set of smooth maps between spaces in the diffeological sense. 
 Natural structures are given to them via the basic constructions {\itshape pull-back} and {\itshape push-forward} 
 of diffeologies from a space $(X,\mathcal D)$ and maps $f: Y\to X$ and $g:X \to Z$, respectively. 
 The pull-back diffeology  $f^*(\mathcal D)$ along $f$ (resp. push-forward $g_*(\mathcal D)$ through $g$) is defined as
  the coarsest diffeology on $Y$ (resp. the finest diffeology on $Z$) such that $f$ (resp. $g$) is smooth. 
  The plots for these diffeologies are specified as follows: 

\begin{itemize}  

  \item[i.] A parametrization $Q: V\to Y$ is a plot in $f^*(\mathcal D)$ if and only if $f \circ Q$  is a plot for $X$. As is shown in the following diagram:  

 $$
\xymatrix
{
V \ar[d]_Q \ar@{-->}[rd]^{f\circ Q} &  \\
Y \ar[r] _f & X 
} 
$$
 
 \item[ii.] A parametrization $R: W\to Z$ is a plot in $g_*(\mathcal D)$ if and only if, in every point $w$ of $W$, either $R$ is locally constant or there is a neighborhood $U$ of $w$ and a plot $P: U\to X$ such that $R|_U = g\circ P$. In the latter case we shall say that $P$ \emph{locally lifts} $R$ along $g$. As shown in the following diagram:  

 \end{itemize} 

$$
\xymatrix
{
 U \ar@{-->}[rd] \ar[d]_P \ar@{^(->}[r] & W\ar[d]^R  \\
 X \ar[r]_g & Z
}
$$

The monomorphisms of $\bf{Diff}$ are called {\itshape inductions}: injective maps $i$ where the diffeology of the domain coincides with the pull-back along $i$ of the diffeology of the range;
and the epimorphisms {\itshape subductions}: surjective maps $p$ where the diffeology of the range coincides with the push-forward through $p$ of the diffeology of the domain.
The isomorphisms of $\bf{Diff}$ are called \emph{diffeomorphisms}. They are smooth invertible maps whose inverse function is again smooth. 

The \emph{subspace diffeology} is defined as the pull-back diffeology along the inclusion, the \emph{quotient diffeology} as the push-forward diffeology through the canonical projection, and diffeologies for products and coproducts (= diffeological sum) as the finest (resp. the coarsest) diffeology such that all the projections (resp. inclusions) are smooth. 

The \emph{functional diffeology} for the set of smooth maps $C^\infty(X,X')$ between two diffeological spaces may be defined 
as the coarsest diffeology such that the evaluation map 
$$
ev: C^\infty(X,X')\times X \to X' \text{ \hspace{0.3cm} given by \hspace{0.3cm} } ev(f,x) = f(x)$$
is smooth. 

\medskip
For a smooth surjection $\pi: X \to B$ between diffeological spaces $X,B$ it is possible to define the
 \emph{structure groupoid} $K_\pi$ of the surjection $\pi$ as the groupoid whose objects are $\mathrm{Obj}(K_\pi)=B$ 
 and its arrows are $\mathrm{Mor}(K_\pi)=\bigcup_{b,b'\in B} C^\infty (X_b, X_{b'})$, where  $X_b$ denotes the fiber  $\pi^{-1}(b)$. 
 See \cite[Postulate 8.2]{PIZ13}.
 
 Let $s$ and $t$ denote the usual \emph{source} and \emph{target} maps of an arrow $\varphi \in \mathrm{Mor}(K_\pi)$
 given by $t(\varphi) = \pi(rank(\varphi))$ and $s(\varphi) = \pi(dom(\varphi))$. A diffeology for the groupoid $K_\pi$ is defined as the coarsest diffeology such that 
 the evaluation map $ev: X_s\to X$, given by $ev(\varphi,x)=\varphi(x)$, is smooth, 
 where 
 $$X_s = \{(\varphi, x)\in \mathrm{Mor}(K_\pi)\times X: x \in dom(\varphi)\}$$ 
 denotes the total space of the pull-back of $\pi$ along the map $s$.
  
  The groupoid structure comes with a smooth map called the \emph{characteristic map}:
\begin{equation}
  \label{fb:charctMap}
  \begin{array}{cccc}
    \chi:& \mathrm{Mor}(K_\pi) &\to & B\times B\\
    \quad   &\varphi &\mapsto & (s(\varphi), t(\varphi)) 
  \end{array}
\end{equation}
which, as we shall see in the next section, encodes the properties of the fibers of $\pi$. 

\bigskip

\begin{rem}
\label{rem:d-top&loc}
  {\upshape 
  The diffeological structure over a space naturally induces a topology, the so-called $D$-topology, 
  whose open sets $A\subset X$ are precisely the subsets whose pre-images $P^{-1}(A)$ under every plot $P$ are open subsets 
  of some $\mathbb R^n$. The $D$-topology over a diffeological space is the coarsest topology such that every plot is continuous.
  
  The existence of the $D$-topology allows us to talk about local diffeological properties, defined in $D$-neighborhoods. 
  A map $f: X \to X'$ between diffeological spaces is said to be \emph{locally smooth} in $x\in X$ if there is a 
  $D$-open subset $A$ containing $x$ such that the map $f|_A$ is smooth. That is, if for every plot $P$ of $X$ 
  the parametrization $f\circ P|_{P^{-1}(A)}$ is a plot of $X'$. The map $f: X \to X'$ is smooth if and only 
  if it is locally smooth in every point of $X$. See \cite[Postulate 2.3]{PIZ13}. 
  }
\end{rem}

\subsubsection{Diffeological Fiber Bundles}
\label{SS:DFB}

\begin{definition}
  \label{def:fibrant}
    {\upshape 
    Let $\pi: X \to B$ be a smooth surjection between diffeological spaces. 
    Then $\pi$ is a \emph{diffeological fibration} if and only if the characteristic map $\chi$
    of the structure groupoid $K_\pi$ (see Eq. \ref{fb:charctMap}) is a subduction. 
    In such case, $K_\pi$ is called \emph{fibrating}.
    }
\end{definition}

\medskip

A morphism between surjections $\pi: X \to B$ and $\pi': X' \to B'$  is a pair of smooth maps $(\Phi, \phi)$ 
making the following diagram commute.
$$
\xymatrix
{
X \ar[d]_-\pi \ar[r]^\Phi & X' \ar[d]^-{\pi'}   \\
B \ar[r]^\phi & B' }
$$
Whenever $\Phi$ and  $\phi$ are diffeomorphisms we say that $\pi$ and $\pi'$ are \emph{equivalent}. 
In particular, $\pi$ is \emph{trivial} with diffeological fiber $F$ if it is equivalent to the 
projection $\mathrm{pr}_1:B\times F\to B$. 
We say that $\pi$ is \emph{locally trivial} 
with fiber $F$ if there exists a $D$-open cover $\{U_i\}_{i\in\Lambda}$ of $B$ 
such that the restrictions $\pi|_{U_i}:\pi^{-1}(U_i)\to U_i$ are trivial surjections with fiber $F$.  

\medskip

When $\pi:X\to B$ is a diffeological fibration the surjectivity of $\chi$ (\ref{fb:charctMap}) implies that the 
preimages $X_b=\pi^{-1}(b)$ are all diffeomorphic: they are the fibers $F=X_b$.

\medskip

\begin{lemma}\cite[Postulate 8.9]{PIZ13}\label{lemma:lt.a.plots}
Let $X$ and $B$ diffeological spaces. A smooth map $\pi:X\to B$ is a fibration if and only if there exists a diffeological
 space $F$ such that the pullback $P^{*}(X)$ of $\pi$ along any plot $P:U\to B$ of $B$ is locally trivial with fiber $F$, where 
$$P^{*}(X):=\{(u,x)\in U\times X : P(u)=\pi(x)\}.$$
Thus $\pi$ is \emph{locally trivial along the plots} of $B$.
\end{lemma}

\bigskip

 A group $G$ is a {\itshape diffeological group} if the multiplication and the inverse mappings are smooth:
$$
\begin{array}{cccc}
\mathrm{mul}&=&[(g,g')\to g\cdot g']&\in C^\infty (G\times G,G)\\
 \mathrm{inv}&=&[g\to g^{-1}]&\in C^\infty (G,G).
\end{array}
$$
We say that the group $G$ acts smoothly on a diffeological space $X$ if the group homomorphism $h: G \to {\bf Diff}(X)$ is smooth. The \emph{action map} of $G$ over $X$ is given by:
$$
\begin{array}{ccccc}
\mathcal{A}_{X}^{G}:& X \times G&\to&X\times X & \\
 &(x,g)&\mapsto&(x,g_X(x))&\quad\text{with}\quad g_X=h(g).
\end{array}
$$
\begin{prop}\cite[Postulate 8.11]{PIZ13} Consider the projection $\pi:X\to X/G$. If $\mathcal{A}_{X}^{G}$ is an induction then $\pi$ 
  is a diffeological fibration with fiber $G$ and the action of $G$ on $X$ is called \emph{principal}. Moreover, any equivalent
   fiber bundle $\pi':X\to Q$ to $\pi$ is a $G$-\emph{principal fiber bundle} or $G$-\emph{principal fibration}. 
\end{prop}


\subsubsection{The irrational torus}
\label{SS:irratorus}
The covering $\mathrm{pr}:\mathbb{R}^2\to\mathbb{T}^2=\mathbb{R}^{2}/\mathbb{Z}^2$ is clearly a $\mathbb{Z}^2$-principal fibration 
and the torus $\mathbb{T}^2$ is a diffeological group. Then any line $L_\alpha$ of slope $\alpha$ in $\mathbb{R}^2$ 
through the origin is projected onto a 
diffeological one-parameter subgroup $\mathcal{S}_\alpha$ of 
$\mathbb{T}^2$. Then it induces a group action whose action map $\mathcal{A}_{\mathbb{T}^2}^{\mathcal{S}_\alpha}$ is an induction. 
When $\alpha$ is an irrational number we get that $\mathcal{S}_\alpha$ is isomorphic to $\mathbb{R}$ and it is called the \emph {irrational solenoid} of slope $\alpha$.
 Thus, the projection
\begin{equation}\label{projection:solenoid}
 \pi_\alpha:\mathbb{T}^2 \to \mathbb{T}^2/\mathcal{S}_\alpha=\mathbb{T}_\alpha
\end{equation}  
is a $\mathbb{R}$-principal fiber bundle over the \emph{irrational torus} $\mathbb{T}_\alpha$ known as the Kronecker flow of slope $\alpha$ in $\mathbb{T}^2$.
  
\medskip

Consider the subgroup  $\mathbb{Z} + \alpha \mathbb{Z}$ in $\mathbb R$ with $\alpha$ irrational. 
It is a diffeological group and its action map $\mathcal{A}_{\mathbb{R}}^{\mathbb{Z} + \alpha \mathbb{Z}}$ is an induction. 
Then, the quotient   
\begin{equation}\label{projection:cover}
\mathrm{pr}_\alpha:\mathbb{R}\to \mathbb R / \mathbb Z + \alpha \mathbb Z=\mathbb{T}_\alpha
\end{equation}
is a diffeological space and it is diffeomorphic to $\mathbb{T}^2/\mathcal{S}_\alpha$ (see \cite[Proposition 3]{PIZ20}).  
The density of $\mathbb Z + \alpha \mathbb Z$ in $\mathbb R$ implies that $\pi_\alpha$ is locally trivial 
along the plots of $\mathbb{T}_\alpha$ (see Lemma \ref{lemma:lt.a.plots}) but it is not locally trivial.

\medskip

For two irrational numbers $\alpha,\beta$, the following lemma from \cite{PIZ20} 
dictates the conditions for 
 the irrational tori $\mathbb{T}_\alpha, \mathbb{T}_\beta$ to be diffeomorphic.

\medskip

\begin{lemma}\label{lemma:IrracionalToriEquivalence}
    
    Let $\alpha$ and $\beta$ be irrational numbers. The following assertions hold:
   \begin{itemize}
   \item[i)]There is a non-trivial smooth map $f \in C^\infty(\mathbb{T}_\alpha, \mathbb{T}_\beta)$ if and only if there are integers $a,b,c,d$ such that 
    $$ \beta = \frac{a\alpha - c }{d - b\alpha},\quad\text{with}\quad ad-bc\neq0. $$ 
   \item[ii)] $f$ lifts through the projections $\mathrm{pr}_\alpha,\mathrm{pr}_\beta$ (Eq. \ref{projection:cover}) to an affine map $F(x)=\lambda x+\mu$ with $\lambda=a+b\beta$ and $\lambda\alpha=c+d\beta$. 
   \item[iii)] $f$ is a diffeomorphism if and only if $ad-cb = \pm 1$.
   \item[iv)] If (iii) hold, the diffeomorphism $f$ induces an equivalence between the \mbox{$\mathbb{R}$-principal} fibrations $\pi_\alpha$ and $\pi_\beta$.
   \end{itemize}
 \end{lemma}

\begin{proof}
  
  Let $f:\mathbb{T}_\alpha \to \mathbb{T}_\beta$ be a smooth map. Then for every $x_0\in\mathbb{R}$ there exists an open interval $V$ centered at $x_0$ and a smooth map $F:V\subset \mathbb R \to \mathbb R$ 
  such that 
  \begin{equation}\label{eq:definitionLocalLift}
  \mathrm{pr}_\beta \circ F= f\circ\mathrm{pr}_\alpha. 
  \end{equation} 
  Consider $\epsilon>0$ and an interval $U\subset V$ centered at $x_0$ such that for every $x\in U$ and for all $n,m \in \mathbb Z$ 
  satisfying that $n+m\alpha\in (-\epsilon,\epsilon)$ then $x+n+m\alpha \in V$. By Equation \ref{eq:definitionLocalLift} we have that  
  \begin{equation}
    \label{eq:toroalfatorobeta}
    F(x+n+m\alpha) = F(x) + n'+ \beta m',
  \end{equation}
  for $x\in U$ and $n+m\alpha$ fixed. The irrationality of $\alpha, \beta$ implies that $n'$ and $m'$ are unique. 
  But $F$ is smooth, then its derivative with respect to $x$ at $x_0$ gives us 
  \begin{equation}
    F'(x_0+n+m\alpha) = F'(x_0).
  \end{equation}
  But $\mathbb Z + \alpha \mathbb Z \cap (-\epsilon, \epsilon)$ is dense in $(-\epsilon, \epsilon)$ 
  and $F'$ is continuous, then $F'(x) = F'(x_0)$  for every $x \in U$. 
  Thus $F|_U$ is affine, that is, there exist $\lambda, \mu \in \mathbb R$
  such that $F(x) = \lambda x +\mu $ for every $x$ in $U$. As $\mathbb Z + \alpha \mathbb Z$ is dense in $\mathbb R$, 
  $\mathrm{pr}_\alpha(U) = \mathbb{T}_\alpha$. Therefore the map $F$ completely defines $f$ and extends to the covers. 

  Using Equation \ref{eq:toroalfatorobeta} we get that
  \begin{equation}
    \lambda(x+n+\alpha m)+\mu = \lambda x + \mu + n' +\beta m',  
  \end{equation}
  for every $n+\alpha m \in (-\epsilon,\epsilon)$. Thus, we get that
  $\lambda (\mathbb Z + \alpha \mathbb Z) \subset \mathbb Z + \beta \mathbb Z$. Therefore $\lambda = a+\beta b, \lambda\alpha = c+\beta d$,
  and $$\alpha =\frac{c+\beta d}{ a+\beta b}.$$ 

  Considering $\mathbb Z + \alpha \mathbb Z, \mathbb Z + \beta \mathbb Z$ as $\mathbb Z$ 
  modules, the smooth map $F$ is an isomorphism if and only if $ad-cb= \pm 1$. 
 
  For the quotients $\mathbb T^2 / \mathcal S_\alpha,\mathbb T^2 / \mathcal S_\beta$, 
  it is easy to see that the matrix element 
  \begin{equation}
    M=\begin{pmatrix}
      d & -b \\
      -c & a 
      \end{pmatrix}
  \end{equation}
  is a linear transformation of $\mathbb R^2$ that fixes the $\mathbb Z^2$ lattice and maps the line 
  $L_\alpha$ onto $L_\beta$. And whenever $\det(M)=\pm1$, the linear transformation $M$ induces an equivalence between 
  $\mathbb R$-principal fiber bundles  
  $$
  \xymatrix
  {
  \mathbb T^2 \ar[d]_-{\pi_\alpha} \ar[r]^{\Phi_M} & \mathbb T^2 \ar[d]^-{\pi_\beta}   \\
  \mathbb{T}_\alpha \ar[r]^\phi & \mathbb{T}_\beta }
  $$
  For a more complete treatment we refer the reader to \cite[Propositions 4-7]{PIZ20}.

\end{proof}

In particular, the orbit of $L_\alpha$ under the action of the group $\mathrm{GL}(2,\mathbb{Z})$ is a family of 
irrational tori diffeomorphic to $\mathbb{T}_\alpha$. 
The group of connected components of $C^\infty(\mathbb{T}_\alpha, \mathbb{T}_\alpha)$ 
is a subgroup of $\mathrm{GL}(2,\mathbb{Z})$ 
and it is different from the group $\{\mathrm{id},\mathrm{-id}\}$ if and only if $\alpha$ is a quadratic irrational 
(see Proposition 3.3 in \cite{PIZ85}).  

\bigskip

Another property worth mentioning is that the exact sequence of homotopy
 for fiber bundles also holds in the diffeological context, 
 (see \cite[Postulate 8.21]{PIZ13}). 
In our case, it gives us the following exact sequence:  
\begin{equation}\label{exactHomotopi:kronecker}
\cdots \longrightarrow \pi_1(\mathbb R) \longrightarrow \pi_1(\mathbb T^2) 
\longrightarrow \pi_1(\mathbb T_\alpha ) \longrightarrow \pi_0(\mathbb R) \longrightarrow \cdots.
\end{equation}
The diffeological space $\mathbb R$ is connected and simply connected, 
then the fundamental groups of $\mathbb T^2$ and $\mathbb T_\alpha$ are isomorphic.


\section{Diffeological Tiling Spaces}
\label{sec:diff-TS}

In the present section we introduce a diffeological structure for the tiling spaces $\Omega$. 
For now, we let $\Omega$ be any space of tilings of $\mathbb R^d$.

\subsection{The Diffeological Space $\Omega$}
\label{SS:DSomega}

\medskip
\begin{definition}
\label{def:diff-omega}
{\upshape
 Define a diffeology  $\mathcal D$  for $\Omega$ as the finest diffeology such that 
 the translation action of $\mathbb R^d\times \Omega\to \Omega$ given by $(x,T)\mapsto T+x$ is smooth.   
}
\end{definition}

\medskip

Observe that, for any diffeology $\mathcal D'$ on $\Omega$ such that the action of $\mathbb R^d$ 
by translation is smooth, the \emph{orbit maps}
$\hat T: \mathbb R^d \to \Omega$ given by $\hat T(x) = T+x$ are smooth for every $T$. 
This means that any such diffeology 
$\mathcal D'$ contains every orbit map $\hat T:\mathbb R^d \to \Omega$ as a plot, and thus 
the diffeology generated by the family of orbit maps is contained in $\mathcal D'$. 
Let
\begin{equation}
  \label{eq:family-orbit-maps}
\mathcal F = \{ \hat T : \mathbb R^d \to \Omega :  T\in \Omega \text{ and } \hat T(x) = T + x \}. 
\end{equation}
and $\mathcal D_1 = \langle F \rangle$. We have that the diffeology $\mathcal D$ from 
Definition \ref{def:diff-omega} contains $\mathcal D_1$. But $\mathcal D$ is minimal 
by definition of generating family \cite[Postulate 1.68]{PIZ13}. Thus $\mathcal D = \mathcal D_1$. 





\medskip

By definition of generating family, the plots of 
$\mathcal D = \langle F\rangle$ are specified 
as the parametrizations $P: U \to \Omega$ such that for every $u\in U$ there exists an open neighborhood 
$V\subset U$ of $u$ such that either $P|_V$ is a constant parametrization 
or there exists an element 
$\hat T: \mathbb R^d \to \Omega$ of $\mathcal F$ and a smooth map $Q:V\to \mathbb R^d$ 
such that $P|_V = \hat T \circ Q$. 
As shown in the following diagram. 
$$
\xymatrix{
V \ar@{-->}[r]^Q \ar@{^(->}[d]  &  \mathbb R^d \ar[d]^{\hat T}  \\
U \ar[r] ^P & \Omega   }
$$


\medskip

With regard to equivalence we note that diffeomorphisms 
of tiling spaces map orbits onto orbits, and vice versa, an \emph{orbit equivalence} between tiling spaces is a diffeomorphism.  

\medskip
\begin{prop}
\label{prop:conjugadossyssdifeomorfos}

Let $(\Omega, \rho)$ and $(\Omega', \rho')$ be two $d$-dimensional tiling spaces endowed with their 
respective $\mathbb R^d$-actions $\rho$ and $\rho'$ by translation. 
Then, they are orbital equivalent if and only if they are diffeomorphic. 
\end{prop}

\begin{proof}
    Let $\varphi:\Omega \to \Omega'$ an orbital equivalence between tiling spaces and $P:U \to \Omega$ a plot. Up to pre-composing with a smooth map $f$, 
    $P$ is locally a translation. This means that for every $u \in U$ there exists an open neighborhood $V\subset U$ of $u$, 
    a tiling $T \in \Omega$ and a smooth map $f: V\to \mathbb R^d$ such that $P|_V = \hat T \circ f$. Hence
    $$
    \varphi  \circ P  (u)  =  \varphi \circ \hat T \circ f (u) = \varphi ( T + f(u)) = \varphi(T) + y =  \hat{\varphi(T)} ( y ), 
    $$
    for some $y\in \mathbb R^d$. The third equality is given by hypothesis and $\hat {\varphi(T)}$ is an element of 
    the generating family $\mathcal F'$ for the diffeology of $\Omega'$. 
    Thus $\varphi \circ P$ is a plot of $\Omega'$ and then $\varphi$ is smooth. 
    The same argument works for proving the smoothness of $\varphi^{-1}$.
    Hence $\varphi$ is a diffeomorphism.

Conversely, let $\varphi : \Omega \to \Omega'$ be a diffeomorphism. 
Let $T\in \Omega$ and $x \in \mathbb R^d$. Consider the path $\gamma: \mathbb R \to \Omega$ 
given by $\gamma(s) = (1-s) T + s( T+x)$. This path joins $T$ to $T+x$. 
Then, as $\varphi$ is smooth, $\varphi \circ \gamma$ is a path in $\Omega'$ such that 
$\varphi\circ \gamma (0) = \varphi (T)$ and $\varphi\circ \gamma (1) = \varphi (T+x)$. 
Thus $\varphi(T)$ and $\varphi(T+x)$ lie in the same path component of $\Omega'$, so $\varphi(T+x)$ 
must be a translate of $\varphi(T)$. Hence $\varphi$ maps orbits in $\Omega$ onto orbits 
in $\Omega'$. 

\end{proof}

\bigskip

\smallskip

\begin{prop}
\label{prop:sust-is-diffeo}
If $\omega$ is a primitive substitution rule for $\Omega$ then the map $\omega: \Omega \to \Omega$ is smooth. 
Moreover, if $\Omega$ contains at least one non-periodic tiling then $\omega$ is a diffeomorphism. 
\end{prop} 

\begin{proof}
Let $P:U\to \Omega$ be a plot for $\Omega$. We must show that $\omega\circ P$ is also a plot. 
By definition, for every point $u\in U$ there is a 
neighborhood $V\subset U$, an element $\hat T:  \mathbb R^d \to \Omega$ 
of the generating family of the diffeology $\mathcal D$ and a smooth map $Q:V\to W$  such that $P|_V = \hat T\circ Q$ (assuming $P$ non constant, as if it were the result would follow).
$$
\xymatrix{
U \ar[r]^P  & \Omega \ar[r]^\omega & \Omega \\
V \ar@{^(->}[u]  \ar[r] ^Q & W \ar[u]^{\hat T} \ar[ur]_{\hat {\omega(T)}} }
$$
Then for every $v\in V$ and some $T\in \Omega$ we have that
$$
\omega \circ P (v) = \omega \circ \hat T \circ Q (v) = \omega \circ \hat T(x) = \omega (T+x) =  \omega(T) + \lambda \cdot x = 
\hat {\omega(T) }(\lambda\cdot x), 
$$
where  $x= Q(v)$ and $\lambda$ is the expansion factor of $\omega$. 
The map $\hat {\omega(T) }$ is an element of $\mathcal F$ because $\omega(T) \in \Omega$. It follows that every $u$ in $U$ has a neighborhood $V$ such that
 $\omega\circ P|_V = \hat {\omega(T) } \circ Q$. Then $\omega \circ P$ is a plot for $\Omega$ and $\omega$
 is smooth.
 
\smallskip

The fact that $\omega$ is invertible whenever $\Omega$ 
contains a non-periodic tiling follows from Theorem 1.3 in \cite{sadun08}, and the smoothness 
of $\omega^{-1}$ is verified in the same way as before.  
\end{proof}

\medskip

Recall that the topological space $\Omega$ is locally homeomorphic to a cartesian product of a ball $D$ times totally disconnected set $K$. 
Note that by the definition of $\mathcal D$ we know that it is possible to move smoothly within the factor $D$, 
while the previous proposition means that there is also a smooth dynamic within the factor $K$. 

\medskip
 
\begin{rem} 
\upshape{
Note that there are at least two natural topologies over $\Omega$: the $D$-topology 
and the topology induced by the metric $d$. 
It is not difficult to prove that the former is finer
than the latter.
}
\end{rem}

\medskip

\subsection{The bundle $\tilde p:\Omega \to  \mathbb{T}_\alpha$}

In the current section will discuss on how to bring Sadun-Williams' (Theorem \ref{thm:SW-bundle}) 
construction into the diffeological setting. 
For simplicity, we consider first an unidimensional tiling space $\Omega$. 
Recall from Theorem \ref{thm:APinverselimits} that $\Omega$ is conjugate --and hence diffeomorphic following 
Proposition \ref{prop:conjugadossyssdifeomorfos}-- to the inverse limit of branched curves $\Omega_n$
$$
\Omega = \lim_{\leftarrow n} \{\Omega_n, f_n\}.$$ 

The construction of Theorem \ref{thm:SW-bundle} relies on Theorem \ref{thm:APinverselimits} 
as it uses the maps between $\Omega$ and its aproximant $\Omega_0$ in order to obtain
a (topological) fiber bundle with total space $\Omega$, the one-dimensional torus $\mathbb S^1$ as base space, 
and fibers homeomorphic to the Cantor set $K$. 
Consider such fibration
$$
K\hookrightarrow \Omega \stackrel{p}{\to} \mathbb S^1.
$$
The map $p$ is the composition of maps
$$
\Omega \stackrel{\pi_0}{\longrightarrow}  \Omega_0  \stackrel{q}{\longrightarrow} \mathbb R \stackrel{\pi}{\longrightarrow} \mathbb S^1,$$
where 
$\pi_0: \Omega \to \Omega_0$ is the projection onto the first factor of $\Omega$ as an inverse limit,
$q$ is the function that maps the cells of $\Omega_0$ onto intervals $[n,n+1]\subset \mathbb R$, 
and $\pi$ is the quotient map $\pi : \mathbb R\to \mathbb R/\mathbb Z$.  
Thus the function $p$ maps the tiling $T$ to the position of its origin in the tile it belongs to, projected onto $\mathbb S^1$. 

\smallskip

In the diffeological context, it is more natural to consider the following alternative map:  
 let $\{a_i\}$ be the set of lengths 
 of the tiles of $T$. Then, instead of the previous quotient map $\pi: \mathbb R \to \mathbb S^1$, 
 we map onto the quotient space $\mathbb R/(a_1\mathbb Z + \ldots +  a_n\mathbb Z)$, 
 by mapping the vertices of a tiling $T$ to $a_1\mathbb Z + \ldots +  a_n\mathbb Z$, instead of $\mathbb Z$. 
 Note that this is precisely a \emph{return module}.

\medskip

\begin{definition}
  \label{def:returnmodule}
  {\upshape Let $P$ be a patch of a tiling $T$ of $\mathbb R^d$. Define the \emph{return vectors} of $P$ as the set of vectors 
  $v$ mapping $P$ to other occurences of $P$ in the tiling, \emph{i.e.} the vectors $v$ such that the patern 
  at $P+v$ coincides with $P$. 
  Define the \emph{return module} $R_P$ of the patch $P$ as the $\mathbb Z$-span of the return vectors. 
  }
\end{definition}

\medskip

So, for a given tiling space $\Omega$ in $\mathbb R^d$, we pick a tiling $T\in \Omega$ and a vertex $v$ on it. 
Let $\mathbb T_v = \mathbb R^d/ R_v$, where $R_v$ is the return module of the vertex $v$. 
We define $\tilde p$ as the composition of maps 
$$
\Omega \stackrel{\hat q}{\longrightarrow} \mathbb R^d \stackrel{\mathrm{pr}_v}{\longrightarrow} \mathbb{T}_v$$
In this case $\hat q$ is the function that maps a tiling $T$ onto the position of its origin in $\mathbb R^d$, 
and $\mathrm{pr}_v$ is the quotient map $\mathrm{pr}_v : \mathbb R^d\to \mathbb T_v$. 
Thus $\tilde p$ is the function that maps a tiling $T$ to the equivalence class of the location of its origin 
in the quotient space $\mathbb{T}_v$. Note that in the familiar case where $d=1$ and the tile set is $\{1,\alpha\}$, 
the return module $R_v$ coincides with $\mathbb Z + \alpha \mathbb Z$ and so the base space of the function 
$\tilde p$ is the irrational torus $\mathbb T_\alpha$. 


\medskip

\begin{lemma}

With the diffeology $\mathcal D$ for $\Omega$ from (\ref{def:diff-omega}), the standard diffeology for $\mathbb R^d$, 
and the quotient diffeology for $\mathbb T_v$, the map $\tilde p: \Omega \to \mathbb{T}_v$ is a smooth surjection.  
\end{lemma} 

\begin{proof}
The map $\tilde p$ is clearly surjective as, for a given $\bar z \in \mathbb{T}_v$ represented by 
$z+\sum_{i=1}^n v_i n_i$ with $z \in \mathbb R, n_i \in \mathbb Z$, and $\{ v_i \}$ the return vectors of a given vertex in $\Omega$, 
there is a tiling $T\in \Omega$ whose origin lies in $z$ that is mapped to $\bar z$. 
For proving smoothness note that the map $\mathrm{pr}_v$ is smooth by definition and 
every plot $P$ of $\Omega$ is locally a translation, so $\hat q\circ P$ is also a translation in $\mathbb R^d$ 
and translations are smooth, so it is a plot of $\mathbb R^d$.   
 Hence $\tilde p $ is smooth.
\end{proof}


\subsubsection{The structure groupoid for $\tilde p: \Omega \to \mathbb{T}_v$. }

We proceed to investigate the structure groupoid associated to the smooth surjection $\tilde p: \Omega \to \mathbb{T}_v$. Recall from \ref{SS:DFB}. 
 that $G_0 = \mathbb{T}_v$ and $$ G_1 = \bigcup_{z,z'\in G_0} \mathbf{Diff} (\Omega_z, \Omega_{z'}),$$
for $\Omega_z := \tilde p^{-1}(z)$. 
 The maps \emph{source} $s:G_1 \to G_0$ and \emph{target} $t:G_1 \to G_0$ are defined by $s(\varphi) = \tilde p(dom(\varphi)) $ and $ t(\varphi) = \tilde p(ran(\varphi))$. 

A diffeology for the groupoid $(G_1,G_0)$ is defined as the coarsest diffeology such that the evaluation map $ev: X_s\to \Omega$, given by $ev(\varphi,T)=\varphi(T)$, from the pull-back $$X_s = \{(\varphi, T)\in G_1\times\Omega: T \in dom(\varphi)\}$$ along the map $s$, is smooth and makes $(G_1,G_0)$ a diffeological groupoid. 

\bigskip

\begin{prop}
The map $\tilde p: \Omega \to \mathbb{T}_v$ is a diffeological fibration.
\end{prop}

\begin{proof} 

We must show that the structure groupoid $(G_1,G_0)$ associated to $\tilde p$ is fibrating, \emph{i.e.} 
the characteristic map $\chi: G_1 \to \mathbb{T}_v \times \mathbb{T}_v $ is a 
subduction\footnote{See discussion after Example \ref{ex:standard-diff}.}. 
In order to conclude that, we have to show that for every plot $Q:U \to \mathbb{T}_v\times \mathbb{T}_v$ 
and every point $u\in U$ there exists a neighborhood $V$ of $u$ such that either $Q|_V$ is constant or there exists a plot $P: V\to G_1$ that locally lifts $Q$. 

The plots of $G_1$ are specified as those parametrizations $P$ such that the maps $\chi\circ P$ and 
$$
P_s : \Omega_{s\circ P }\to \Omega \text{  given by  } P_s(u, T) = P(u)(T) 
$$

$$
P_t : \Omega_{t\circ P }\to \Omega \text{  given by  } P_t(u, T) = P(u)^{-1}(T),
$$
where $\Omega_{s\circ P } $ stands for the pull-back of $\tilde p$ along $s\circ P$, and analogously $\Omega_{t\circ P }$, are smooth. 

Let $Q:U\to \mathbb{T}_v \times \mathbb{T}_v$ be a plot for the product diffeology on $\mathbb{T}_v \times \mathbb{T}_v$, 
then $Q(u) = (Q_1(u),Q_2(u))$ where $Q_i$ a plot of $\mathbb{T}_v$. 
Moreover each $Q_i$ lifts locally (say in neighborhoods $V_i$, with $V_1 \times V_2 := V \subset U$) to a plot 
$\bar Q_i$ of $\mathbb R^d$ because the  diffeology on $\mathbb{T}_v$ coincides with the quotient diffeology 
$\mathbb R^d/R_v$. 

Let $u\in U$, $z_s = \bar Q_1(u)$, $z_t = \bar Q_2(u), \bar z_s = \mathrm{pr}_v(z_s)$ and $\bar z_t = \mathrm{pr}_v(z_t)$. 
Then $z_s,z_t \in \mathbb R^d$ and $\bar z_s, \bar z_t \in \mathbb{T}_v$. Let $\varphi_{z_s,z_t}: \Omega \to \Omega$ be the map given by 
$$ \varphi_{z_s,z_t}(T)= T-z_s + z_t. $$ 
Note that this map translates the fiber 
$\tilde p^{-1} (\bar z_s):=\Omega_{\bar z_s}$ to $\tilde p^{-1} (\bar z_t):=\Omega_{\bar z_t}$. 
It is moreover a diffeomorphism between these fibers: it is smooth because it is a 
translation as well as its inverse. Then it is an element of $G_1$ satisfying 
$s(\varphi_{z_s,z_t})=\bar z_s$ and $t(\varphi_{z_s,z_t})=\bar z_t$. 
Thus we define a parametrization $P:V\to G_1$ given by $ P(u) = \varphi_{z_s,z_t} $ and assert 
that it is the plot we needed.

The maps $P_s$, $P_t$ are immediately seen to be smooth because they are translations of tilings, and $\chi\circ P$ equals $(\mathrm{pr}_v \times \mathrm{pr}_v) \circ \bar Q$ which is a composition of smooth maps, and so it is also smooth. Hence $P$ is a plot of $G_1$ and it lifts the plot $Q:U \to \mathbb{T}_v\times \mathbb{T}_v$ by construction. Consequently $\chi$ is a subduction and the groupoid $(G_1,G_0)$ is fibranting.

$$
\xymatrix{
V \ar[r]^P\ar@{_(->}[d] \ar[rd]^{\bar Q} & G_1 \ar@{-->}[d]  \ar@/^/[rdd]^\chi \\
U  \ar@/_/[rrd]_-Q   &  \mathbb R^d \times \mathbb R^d \ar[rd]^-{\mathrm{pr}_v^2}\\
& &  \mathbb{T}_v  \times \mathbb{T}_v  } \\
$$

\end{proof}

\medspace

\begin{theorem-a}
Let $\Omega$ be a space of repetitive tilings of $\mathbb R^d$. 
Then $\Omega$ is a covering space over $ \mathbb R^d / R_v = \mathbb{T}_v$.  
\end{theorem-a}
\begin{proof}
  
  The result is essentially the previous proposition plus the fact that, 
  just as in the topological case, a \emph{covering space} in diffeology is a fiber bundle with discrete fiber. 

\end{proof}

\smallskip

This map has the advantage over Williams' and Sadun-Williams constructions that the proportion of the length of the tiles is preserved. 
  In the one-dimensional case, it provides a necessary condition for diffeomorphism between tiling spaces and thus for 
  orbital equivalence. This is due to the classification of irrational tori \cite{PIZ85}, as we shall see in the following section.

\bigskip


    


\section{Diffeological and sturmian equivalences for Tiling spaces }
\label{section:diff-sturm}

In the present section we study another bundle structure from projection method 
one-dimensional tiling spaces $\Omega_\alpha$
 over $\mathbb T_\alpha$. 
We will bring notions of equivalence between sturmian systems (Lemma \ref{Lemma:EquivSturmian}) 
into the context of tiling spaces through their diffeological structures.  
We will relate explicitly the sturmian sequences of $(X_\alpha,\sigma)$ with 
canonical projection tilings in $\Omega_\alpha$.

\smallskip

For this matter we define a function $\Phi : \Omega_\alpha \to X_\alpha$ which ``forgets" the length of the 
tiles but preserves the combinatorics of every tiling, \emph{i.e.} the funcion $\Phi$ maps the tiling $T$ onto the sequence 
$S\in\{0,1\}^\mathbb Z$ encoding the tiles of $T$, 
where $\{0,1\}$ represent the proto-tile set of $T$. 
This function is not invertible but there is a canonical section
$\Psi : X_\alpha \to \Omega_\alpha$ which assigns to a sturmian sequence $S$ the tiling $\Psi(S)=T$
in the tile set $\{0,1\}$ with lenghts $\frac{1}{\sqrt{1+\alpha^2}}$ and $\frac{\alpha}{\sqrt{1+\alpha^2}}$, 
respectively, 
appearing with the same combinatorics as in $S$, and with its origin placed on the left vertex of 
the tile associated to $s_0$. We define a diffeology for $X_\alpha$ as the pull-back diffeology of $\Omega_\alpha$
along the map $\Psi$. Thus  the map $\Psi$ is automatically smooth. 


\medskip

\begin{lemma}
  \label{lemma:Phiissmooth}
  With the aforementioned diffeology, the map $\Phi$ is smooth. 

\end{lemma}

\begin{proof}
  Let $P:U\to \Omega_\alpha$ be a plot of $\Omega_\alpha$. We must show that 
  $\Phi\circ P : U \to X_\alpha$ is a plot of $X_\alpha$. But this is the case
  if and only if $\Psi\circ\Phi\circ P$ is a plot of $\Omega_\alpha$. 
  As $P$ is a plot of $\Omega_\alpha$, 
  by definition of the generating family $\mathcal F$
  (see Definition \ref{def:diff-omega}) for the diffeology of $\Omega_\alpha$,
  for every $u\in U$ there exists a neighborhood $V\subset U$ containing $u$,
  a smooth map $g: V \to \mathbb R$, and an element $\hat T:  \mathbb R^d\to \Omega_\alpha$ of the generating family, 
  with $T \in \Omega_\alpha$, such that $P|_V = \hat T \circ g$. Then 
  \begin{equation}
    (\Psi\circ\Phi\circ P)|_V (v) = (\Psi\circ\Phi)\circ \hat T \circ g (v)= (\Psi\circ\Phi)(T+g(v)).
  \end{equation} 
  Note that both $T+g(v)$ and $(\Psi\circ\Phi)(T+g(v))$ are elements of $\Omega_\alpha$ with exactly 
  the same combinatorics, so they must be translates of each other. So $\Psi\circ\Phi\circ P$ is 
  locally a translation, and thus a plot of $\Omega_\alpha$. Then $\Phi$ is smooth. 
\end{proof}

\medskip

\begin{lemma}
  \label{lemma:shiftissmooth}
  
  With the aforementioned diffeology for $X_\alpha$ the shift action $\sigma: X_\alpha \to X_\alpha$ is 
  a diffeomorphism. 

\end{lemma}

\begin{proof}
  
  We must show first that the composition of every plot of $X_\alpha$ with the shift map is a plot of $X_\alpha$ itself. 
  Let $P:U \to X_\alpha$ be a plot of $X_\alpha$. By definition of the pull-back diffeology along the map $\Psi$ defined above, 
  the composition $\Psi \circ P$ is a plot of $\Omega_\alpha$. Then, by definition of the generating family $\mathcal F$
  for the diffeology of $\Omega_\alpha$, 
  for every $u \in U$ there exists a neighborhood $V\subset U$ containing $u$, 
  a smooth map $g: V \to W \subset \mathbb R$, and an element $\hat T:  \mathbb R^d\to \Omega_\alpha$ of $\mathcal F$, 
  with $T \in \Omega_\alpha$, such that 
  \begin{equation}
    (\Psi\circ P)|_V = \hat T \circ g.
  \end{equation} 
  This particularly means that a translate of the tiling $T$ is in the image of the map $\Psi$. Let $v \in V$,
    $P(v) = S\in X_\alpha$ and $\Psi(S) = T+x$, with $x = g(v)$. Then
 $$
  \Psi (\sigma(S)) = T + x + \delta_{\alpha},
 $$
  where 
 
 \begin{equation}\label{fnct:delta}
 \delta_{\alpha} = 
  \left\{
	\begin{array}{lll}
		\frac{1}{\sqrt{1+\alpha^2}} & & \mbox{ if  } \sigma(S)_0=0 \\
    \\
	  \frac{\alpha}{\sqrt{1+\alpha^2}} & & \mbox{ if  } \sigma(S)_0=1.
	\end{array} 
  \right.
 \end{equation}
  
Then
$$
  (\Psi \circ \sigma \circ P)|_V(v) = \widehat{(T+ \delta_{\alpha})} (g(v)),
 $$
thus $\Psi \circ \sigma \circ P$ is a plot of $\Omega_\alpha$. Which means that $\sigma \circ P$ is a plot of $X_\alpha$. Therefore $\sigma$ is smooth.

The same argument works for proving the smoothness of $\sigma^{-1}$, whose action is carried onto 
$\Omega_\alpha$ as a translation in the opposite direction. So the $\delta_\alpha$ from above is negative in the present case. 
We then conclude that $\sigma$ is a diffeomorphism. 

\end{proof}

\medskip

An inductive argument gives us that 
   the $\mathbb Z$-action on $X_\alpha$ given by
  $$
  (n,S) \mapsto \sigma^n(S)
  $$
  is smooth.

\bigskip

\begin{lemma}
  \label{lemma:oeiffdiffeo}
  
  Let $X_\alpha$ and $X_\beta$ be sturmian spaces endowed with the diffeology from above. 
  They are orbit equivalent (Definition \ref{def:oecms}) if and only if they are diffeomorphic.
  
\end{lemma}

\begin{proof}
  
  The result is a direct consequence of Proposition \ref{prop:conjugadossyssdifeomorfos}, Lemmas \ref{lemma:Phiissmooth} and 
  \ref{lemma:shiftissmooth}, and the definition of the mappings $\Psi$ and $\Phi$. The actions may be seen in the 
  following diagram

  $$
  \xymatrix{
      X_\alpha \ar[rr]^{\Psi_\alpha} \ar[dr]^{\sigma} \ar[dd]_f & &
      \Omega_{\alpha} \ar[dd] \ar[dr]^{\nu_\alpha} \\
      &  X_\alpha \ar[rr] \ar[dd] & & \Omega_{\alpha} \ar[dd]^g \\
      X_\beta \ar[rr] \ar[dr]_{\sigma}  & &
      \Omega_{\beta} \ar[dr]_{\nu_\beta} \\
      &  X_\beta \ar[rr]_{\Psi_\beta}  & & \Omega_{\beta} \\
  }
  $$
where $\nu_\alpha:=\Psi_\alpha\circ\sigma\circ\Phi_\alpha$, $\nu_\beta:= \Psi_\beta\circ\sigma\circ\Phi_\beta$ 
and $g = \Psi_\beta \circ f\circ \Phi_\alpha$.

\end{proof}

\medskip 

\begin{theorem-b}
  Let $\Omega_\alpha$ a tiling space of canonical projection of dimension 2 to 1 and irrational 
  slope $\alpha$. Then $\Omega_\alpha$ is a $\mathbb R$-principal bundle over 
  the irrational tori with two origins  
  $\mathbb{T}_{\alpha}^{\bullet}$ defined below. 

\end{theorem-b}
\begin{proof}
The function $\Psi$ is a bijection between the orbits under the shift $\sigma$ in $X_\alpha$ and 
the orbits by continuous translation in $\Omega_\alpha$.
 We prove first that the surjection $X_\alpha\to\mathbb{T}_\alpha$ given by 
 $S(\rho)\mapsto \mathrm{pr}_\alpha(\rho)$,
  for $\mathrm{pr}_\alpha$ defined in \ref{projection:cover}, 
  is well defined. 
  Consider $\rho,\rho'\in (0,1)$ such that 
$$\rho-\rho'=e+f\alpha\in\mathbb Z+\alpha\mathbb Z,$$
and let $S(\rho),S(\rho')$ represent the sequences $\{s_n\}_{n\in\mathbb{Z}},\{s'_{n}\}_{n\in\mathbb{Z}}$ respectively. Thus, we have the following equalities
\begin{equation}
  \label{eq:conmushift}
\begin{aligned}
  s_n &= \lceil n\alpha + \rho \rceil  - \lceil (n-1)\alpha + \rho \rceil\\
  s_n &=\lceil (n+f)\alpha + \rho'+e \rceil  - \lceil (n+f-1)\alpha + \rho'+e \rceil\\
  s_n &=\lceil (n+f)\alpha + \rho' \rceil  - \lceil (n+f-1)\alpha + \rho' \rceil\\
  s_n &=\sigma^f(s'_n).
\end{aligned}
\end{equation}
Then each orbit in $X_\alpha$ is projected onto a class of $\mathbb{T}_\alpha$. The mapping 
is smooth because so is $\mathrm{pr}_\alpha$.

But recall from Remark \ref{rem:cpsts} that 
$\Omega_\alpha$ contains an additional orbit corresponding to sequences differing in two consecutive letters. 
 By considering both of these orbits, we obtain a $\mathbb{Z}$-principal fiber bundle 
 $X_\alpha\to \mathbb T_\alpha$ with a singular point with a double fiber. We deal with this singularity 
 in the following way. 
 
Note that the equivalence classes $[x]$ of $\mathbb{T}^{2}/\mathcal{S}_\alpha$ correspond to orbits 
under the action by translation on $\Omega_\alpha$, so they 
are path connected and diffeomorphic to $\mathbb{R}$. 
We define the irrational tori with two origins $\mathbb{T}_{\alpha}^{\bullet}$ as the quotient 
of the punctured torus $\mathbb{T}^{2}-\{e\}$ by the equivalence relation $x\sim y$ 
if and only if $[x]=[y]$ and both $x$ and $y$ are in the same path connected component of $[x]-\{e\}$, 
for $e\in\mathcal{S}_\alpha$ the identity element of the group $\mathbb{T}^2$. 
The class $\mathcal{S}_\alpha$ is divided in two rays, which are diffeomorphic to the real line.
We will show next that this quotient map is actually a diffeological fibration. 


\smallskip

The inclusion $i:\mathbb{T}^{2}-\{e\}\to\mathbb{T}^{2}$ 
induces the following conmutative diagram of smooth maps:
 \begin{equation}
  \label{diag:fibrados}
\xymatrix{
\mathbb{T}^{2}-\{e\} \ar[r]^-{i} \ar[d]_\pi & \mathbb{T}^{2} \ar[d]^{\pi_{\alpha}}\\
\mathbb{T}_{\alpha}^{\bullet}\ar[r]^{i_\bullet}& \mathbb{T}_{\alpha},
}
\end{equation}
where $i_\bullet$ is surjective.
Let $P: V \to \mathbb T_\alpha^\bullet$ be a plot, so $i_\bullet \circ P=:Q$ is a plot of $\mathbb T_\alpha$. 
As $\mathbb T^2$ is a $\mathbb R$-principal bundle over $\mathbb T_\alpha$, by Lemma \ref{lemma:lt.a.plots}  
we know that, for every $v_0\in V$, there is an open subset $U\subset V$ 
such that $Q|_U^*(\mathbb T_\alpha)$ is a trivial bundle. Which means that there is a diffeomorphism $\eta$
$$\eta: U \times \mathbb R \to Q|_U^*(\mathbb T^2).$$ 

\smallskip
\noindent 
Also note that the inclusion $i:\mathbb{T}^{2}-\{e\}\to\mathbb{T}^{2}$ induces an embedding between the pullbacks
$$P|_U^*(\mathbb T^2-\{e\}) \hookrightarrow Q|_U^*(\mathbb T^2). $$

\smallskip
 Next we consider the map
   $$
  \begin{aligned}
    \mathrm{id} \times \exp: & U \times \mathbb R  \to  U\times \mathbb R \\
      & (v,t) \mapsto  (v, \exp(t) )
  \end{aligned}
$$
 and let $\bar\eta = \eta \circ (\mathrm{id} \times \exp): U \times \mathbb R \to Q|_U^*(\mathbb T^2)$.
Note that the image $\bar \eta(U\times \mathbb R)$ of $\bar \eta$   
does not contain the neutral element $e$ and thus 
 $$ \bar \eta(U\times\mathbb R) \subset P|_U^*(\mathbb T^2-\{e\}) \subset Q|_U^*(\mathbb T^2). $$

Now let $U\times J=U\times [a,b]\subset U \times \mathbb R$. 
We have that $ \bar \eta(U\times J) \cap \mathrm{pr}_1^{-1}(v_0)$,
where $\mathrm{pr}_1$ denotes the projection $\mathrm{pr}_1 : P|_U^*(\mathbb T^2) \subset U \times (\mathbb T^2-\{e\}) \to U$, 
is an interval, say $[a_{v_0}, b_{v_0}]$. This means that for every $v \in U$ and every interval $[a,b]$ 
there is an interval $[a_v, b_v]$  such that
$$ \bar \eta(U\times [a,b]) \cap \mathrm{pr}_1^{-1}(v) = [a_v, b_v].$$

\medskip
So we can exhibit sections $s_a$, $s_b : U \to P|_U^*(\mathbb T^2-\{e\})$
given by
$$ 
\begin{aligned}
   s_a(u) &=  \bar \eta(u, a),\\
   s_b(u) &= \bar \eta(u, b)
\end{aligned}
$$
which allow us to define linear diffeomorphisms $\gamma_u: \mathrm{pr}_1^{-1}(u) \to \mathbb R$, such that 
$\gamma_u(s_a(u))=0$ and $\gamma_u(s_b(u))=1$. Then there is a diffeomorphism


$$ 
\begin{aligned}
  \mathbf r:  P|_U^*(\mathbb T^2-\{e\}) & \to  U \times \mathbb R \\
  x & \mapsto ( u , \gamma_u(x) ),
\end{aligned}
$$
where $u=\mathrm{pr}_1(x)$.
Which means that $\pi:\mathbb T^2 - \{e\} \to \mathbb T_\alpha^\bullet$
 is locally trivial along the plots of $\mathbb T_\alpha^\bullet$ 
and, by Lemma \ref{lemma:lt.a.plots}, it is a diffeological fibration.



\smallskip

 Finally, using \ref{eq:conmushift} and \ref{diag:fibrados}, we define the 
$\mathbb{Z}$-principal fiber bundle $X_\alpha\to \mathbb{T}_{\alpha}^{\bullet}$ which can be extended 
through $\Psi$ in order to obtain the following conmutative diagram of smooth maps.

\begin{equation}
\xymatrix{
X_\alpha \ar[r]^-\Psi \ar[rd]_p & \Omega_\alpha \ar[d]^{\pi_{\alpha}^{\bullet}}\ar[r] & \mathbb{T}^{2} \ar[d]^{\pi_{\alpha}}\\
& \mathbb{T}_{\alpha}^{\bullet}\ar[r]^{i_\bullet} & \mathbb{T}_{\alpha}
}
\end{equation}

\smallbreak

We have seen already in Lemma \ref{lemma:lt.a.plots} and Section \ref{SS:irratorus} that the Kronecker flow 
$\mathcal K_\alpha$ is a $\mathbb R$-principal bundle over $\mathbb T_\alpha$ 
and that there is a bijection between $\Omega_\alpha - \{O(T_1)\}$ and $\mathcal K_\alpha$.
So there is only left to prove that 
the new action map $\mathcal A^\mathbb R_{\Omega_\alpha}$ is still an induction. 
But by construction, the additional fiber of $\Omega_\alpha$ is 
smoothly
projected onto the extra point of $\mathbb{T}_{\alpha}^{\bullet}$, 
so if $\mathcal A^\mathbb R_{\mathcal K_\alpha}$
is an induction then $\mathcal A^\mathbb R_{\Omega_\alpha}$ must also be. 
Thus the surjection $\pi_{\alpha}^{\bullet}$ is $\mathbb{R}$-principal fibration.

\end{proof}

\medskip


As a consequence, we are able to compute the fundamental group of $\Omega_\alpha$.

\medskip
\begin{cor}
  \label{cor:gfomega}
The fundamental group of $\Omega_\alpha$ is 
$$
\pi_1(\Omega_\alpha) = \mathbb{F}_2, 
$$
where $\mathbb{F}_2$ is the free group generated by two elements.
\end{cor}

\begin{proof}
The proof relies on the exact sequence of homotopy \cite[Postulate 8.21]{PIZ13} of the diffeological fibration $\pi:\mathbb{T}^{2}-\{e\}\to \mathbb{T}_{\alpha}^{\bullet}$, 
$$
\ldots \longrightarrow \pi_2(\mathbb{T}_{\alpha}^{\bullet})\longrightarrow \pi_1(\mathbb R) \longrightarrow \pi_1(\mathbb{T}^{2}-\{e\}) 
\longrightarrow \pi_1(\mathbb{T}^{\bullet}_\alpha ) \longrightarrow \pi_0(\mathbb{R}) \longrightarrow \pi_0(\mathbb{T}^{2}-\{e\})  $$ 

\noindent
As $\mathbb{R}$ is simply connected, we have that the fundamental groups $\pi_1(\mathbb{T}^{\bullet}_\alpha ),\pi_1(\mathbb{T}^{2}-\{e\})$ are isomorphic.
Note that $\Omega_\alpha$ is $\mathbb{R}$-principal bundle on $\mathbb{T}^{\bullet}_\alpha$, thus applying the exact sequence of homotopy we obtain that 
$$
\pi_1(\Omega_\alpha) = \pi_1(\mathbb{T}^{\bullet}_\alpha) =  \pi_1(\mathbb{T}^{2}-\{e\})=\mathbb{F}_2.
$$
\end{proof}

\smallskip

Theorem B makes it possible to bring the concept of strong orbit equivalence efectivelly into 
the context of tiling spaces. Note that 
the identification $i_\bullet:\mathbb{T}^{\bullet}_\alpha\to\mathbb{T}_\alpha$ induces a surjective morphism 
$\hat{i}_\bullet$ between $\mathbb{R}$-principal fiber bundles $\pi_{\alpha}^\bullet:\Omega_\alpha\to\mathbb{T}^{\bullet}_\alpha$ 
and $\pi_{\alpha}:\mathcal{K}_\alpha\to\mathbb{T}_\alpha$, besides the smooth map 
 $\hat{i}_\bullet\circ\Psi_\alpha:X_\alpha\to\mathcal{K}_\alpha$. 
  
 If $X_\alpha$ and $X_\beta$ are strong orbit equivalent, there is a cocycle map $n:X_\alpha\to\mathbb{Z}$ and a diffeomorphism $h:X_\alpha\to X_\beta$ such that 
  $h(\sigma(S))=\sigma^{n(S)}(h(S))$. 
  Then it is possible lift $h$ through the smooth maps $\hat{i}_\bullet,\Psi_\alpha$ and $\Psi_\beta$ so as to obtain the following commutative diagram.
  $$
  \xymatrix{
    X_{\alpha}\ar@{-->}[drrr] \ar[rr]^{\Psi_\alpha} \ar[dr] \ar[dd]_h & &
    \Omega_{\alpha}\ar[dl] \ar[dd] \ar[dr]^{i_\bullet \circ \pi^\bullet_\alpha} \\
    & \mathcal{K}_{\alpha}\ar[dd] \ar[rr] & & 
    \mathbb{T}_{\alpha} \ar[dd] \\
    X_{\beta}\ar@{-->}[drrr]  \ar[rr] \ar[dr]_{\hat{i}_\bullet\circ\Psi_\beta} & & \Omega_{\beta} \ar[dl]\ar[dr] \\ 
    & \mathcal{K}_{\beta} \ar[rr]_{\pi_\beta} & & \mathbb{T}_{\beta}
  }
  $$
The lift $\tilde{h}:\Omega_\alpha\to\Omega_\beta$ is more than an orbital equivalence. To see this, we define 
another cocycle $m:\Omega_\alpha\times\mathbb{R}\to \mathbb{Z}$ such that $\sigma^{m(T,x)}(\Phi_\alpha(T))=\Phi_\alpha(T+x)$. Then we have the following relation
\begin{equation}
\tilde{h}(T+x)=\tilde{h}(T)+\sum_{i=0}^{m(T,x)}\sum_{j=0}^{n(\sigma^i(S))}\delta_\beta(\sigma^{j}(h(S)))+d(\tilde{h}(T),\Psi_\beta(h(S))),
\end{equation} 
where $S=\Phi_\alpha(T)$, $d$ is the distance function of the metric of $\Omega_\beta$, 
and $\delta_\beta$ as defined in \ref{fnct:delta}.

\medskip

Then we define this stronger equivalence as follows and summarize the results in another theorem. 

\bigskip

\begin{definition}
  {\upshape 
  Two tiling spaces $\Omega_\alpha$, $\Omega_\beta$ are \emph{strong orbit equivalent}
  if their associated sturmian spaces $(X_\alpha,\sigma), (X_\beta,\sigma)$ are. 
  }
\end{definition}

\bigskip

\begin{theorem-c}
  Let $\Omega_\alpha$, $\Omega_\beta$ tiling spaces satisfying the conditions of Theorem B. Then the 
  following conditions are equivalent:
  \item[i)] $\Omega_\alpha$ and $\Omega_\beta$ are strong orbit equivalent.
  \item[ii)]$(X_\alpha,\sigma), (X_\beta,\sigma)$ are strong orbit equivalent.
  \item[iii)] $\mathbb T_\alpha$ and $\mathbb T_\beta$ are diffeomorphic.
  \item[iv)] $\alpha$ and $\beta$ are conjugated under the modular group $\mathrm{GL}(2,\mathbb{Z})$.

\end{theorem-c}

\begin{proof}
The equivalence $i)\Leftrightarrow ii)$ is by definition, $ii)\Leftrightarrow iv)$ is Lemma \ref{Lemma:EquivSturmian},
 and $iii)\Leftrightarrow iv)$ is Lemma \ref{lemma:IrracionalToriEquivalence}.

\end{proof}

\smallskip

Our results provide a classification of one-dimensional tiling spaces as conjugacy classes of irrational numbers
 under the modular group $\mathrm{GL}(2,\mathbb Z)$, 
which implies what we have called \emph{strong orbit equivalence} for tiling spaces. 
This turns out to be an alternative approach to some old results about 
substitution sturmian sequences \cite[Chapter 6]{fogg-arnoux}. 
In 
the present context they are stated as 

\medskip

\begin{cor}
  \label{cor:whithintheclass}
  Within the class of strong orbit equivalent one-dimensional tiling spaces associated 
  to a quadratic irrational slope $\alpha$ 
  there is a substitution tiling space.   

\end{cor}

\begin{proof}

It is known from \cite{fogg-arnoux, masakova2000substitution, harriss06} that a 
cut-and-project tiling admit a substitution if the irrational slope $\alpha$
is a quadratic Pisot number. Moreover, a theorem by Galois \cite{Galois} asserts that quadratic Pisot numbers have purely periodic 
continued fraction. This, together with the well known fact that the continued fraction of conjugates under the modular group differ in 
a finite number of coeficients, implies that there is a Pisot number within this class of conjugation, provided the irrational 
number $\alpha$ is quadratic. This Pisot number is expresed as the continued fraction of the periodic tail of $\alpha$.

\end{proof}



\bmhead{Acknowledgments}

We thank Ana Rechtman, Alberto Verjovsky, \'Angel Zald\'ivar, Jos\'e Aliste-Prieto, 
 Andr\'es Navas, Fabien Durand and Bernardo Villarreal for their support and advice 
 during different stages of this research. 

\backmatter



\end{document}